\documentclass[12pt]{amsart}
\usepackage[usenames,dvipsnames]{xcolor}
\usepackage{amsmath,amssymb,amsfonts,amsthm,pgf,tikz, enumerate}
\usetikzlibrary{shapes,positioning,decorations}
\usetikzlibrary{decorations.pathreplacing}
\usepackage[bookmarks=true,pdfborder={0 0 0}]{hyperref}
\usepackage{soul}
\usepackage{graphicx}

\newtheorem{theorem}{Theorem}[section]
\newtheorem{lemma}[theorem]{Lemma}
\newtheorem{corollary}[theorem]{Corollary}

\theoremstyle{definition}
\newtheorem{obs}[theorem]{Observation}
\newtheorem{definition}[theorem]{Definition}
\newtheorem{example}[theorem]{Example}

\theoremstyle{remark}
\newtheorem{remark}[theorem]{Remark}

\numberwithin{equation}{section}

\DeclareMathOperator{\tr}{tr}
\DeclareMathOperator{\jac}{Jac}
\DeclareMathOperator{\at}{\,{\rule[-3mm]{.1mm}{8mm}}\,}
\usepackage{enumerate} 
\usepackage{bm}

\title[An SISP for Infinite Graphs]{A Structured Inverse Spectrum Problem for Infinite Graphs}
\author{Keivan Hassani Monfared}
\address{University of Calgary}
\email{k1monfared@gmail.com}
\author{Ehssan Khanmohammadi}
\address{Franklin and Marshall College}
\email{ehssan@fandm.edu}

\begin{document}

\begin{abstract}
	It is shown that for a given infinite graph $G$ on countably many vertices, and a compact, infinite set of real numbers $\Lambda$ there is a real symmetric matrix $A$ whose graph is $G$ and its spectrum is $\Lambda$. Moreover, the set of limit points of $\Lambda$ equals the essential spectrum of $A$, and the isolated points of $\Lambda$ are eigenvalues of $A$ with multiplicity one. It is also shown that any two such matrices constructed by our method are approximately unitarily equivalent.
\end{abstract}

\maketitle   
\keywords{Keywords: Inverse Spectrum Problem,  Jacobian Method, Infinite Graph, Spectrum, Essential Spectrum, Approximate Unitary Equivalence. \\


\section{Introduction}
The inverse Sturm-Liouville problem has been studied to a great extent from several points of view. Essentially, it seeks to determine the density of a vibrating string from its natural frequencies (see \cite[p. 83]{borg46}, and \cite{krein51}). Similar but more complicated problems occur in other areas of science and engineering (see for example \cite{backusgilbert67} for problems arising from geophysics, and \cite{reconstructionappliedelectromagnetics} for an example in applied electromagnetics). Common approaches involve solving a finite discretization of the problem and proving that the solution of the discrete problem converges to the solution of the continuous problem. For example, Hald in his 1978 paper \cite{hald78} used the Rayleigh-Ritz method for calculating the eigenvalues of a two point boundary value problem, and reduced the inverse problem for the differential equation to a discrete inverse eigenvalue problem. 

It is of interest to analyze the vibrations of lumped parameter systems, that is, systems that are modelled as rigid masses joined by massless springs. In general, more accurate results are obtained by increasing the number of masses and springs, that is, by increasing the degrees of freedom. As the number of degrees of freedom is increased without limit, the concept of the system with distributed mass and elasticity is formed. Stokey \cite{stokey09} gives a thorough treatment of the direct problem when the degrees of freedom is infinite.

Throughout this paper all matrices have real entries. Let $A$ be an $n\times n$ symmetric matrix and $G$ a simple graph on $n$ vertices. For a set $S \subset [n] = \{1,2, \ldots, n\}$ let $A[S]$ denote the submatrix of $A$ obtained by deleting the rows and columns indexed by $[n] \setminus S$. Also let $G[S]$ denote the induced subgraph of $G$ on the vertices labeled by $S$. We say $G$ is the graph of $A$ when for $i\neq j$ we have $A_{ij} \neq 0$ if and only if vertices $i$ and $j$ are adjacent in $G$. Furthermore, for an infinite matrix $A$ whose rows and columns are indexed by $\mathbb{N}$, we say that an infinite graph $G$ on countably many vertices labeled by $\mathbb{N}$ is the graph of $A$ when $G[S]$ is the graph of $A[S]$ for any nonempty finite $S \subset \mathbb{N}$. Note that $G$ is not necessarily locally finite.

In a lot of cases, one seeks to find a symmetric tridiagonal matrix (a matrix whose graph is a disjoint union of paths) with prescribed eigenvalues \cite{boorgolub78}. In 1989 Duarte~\cite{duarte89} showed that the finite discrete problem has a solution whenever the graph of the matrix is a tree. The $\lambda$-structured-inverse-eigenvalue-problem ($\lambda$-SIEP) asks about the existence of a matrix whose graph and eigenvalues are given. In \cite{lambdamu13} the problem is described and solved when all the eigenvalues are distinct and the graph is finite. 

In this paper our goal is to establish analogous results for when the graph and hence the solution matrix are infinite. In order to do so, in Section \ref{SectionWSP} we introduce a property that captures a notion of \emph{genericity} for finite matrices in the settings of the Jacobian method \cite{lambdamu13}. Then in Section \ref{SectionFiniteLambdaSISP} we solve the finite $\lambda$-SIEP again but instead of solving the problem in one step, we find the solutions $A_n$ using induction on $n$, the number of vertices of the graph, and the Jacobian method. This approach enables us to control the norm of $A_n$'s in each step. Finally, in Section \ref{SecctionInfiniteLambdaSISP} we will show that the limit of $A_n$'s as $n$ approaches infinity exists and has the given graph and spectrum.

Throughout the paper $\circ$ denotes the Schur (entry-wise) product of two matrices, and the Lie bracket $[A, B]$ is the commutator of two matrices $A$ and $B$, that is $AB-BA$. A zero matrix of appropriate size is denoted by $O$, all vectors are written in bold small letters, and $\bm 0$ denotes a zero vector of appropriate size to the context.

\section{The Weak Spectral Property}\label{SectionWSP}
The problem studied in this paper is about the existence of a matrix with a given spectrum and the property that each off-diagonal entry of the matrix is prescribed to be either zero or nonzero with no restriction on the diagonal entries. The Jacobian method~\cite{lambdamu13} starts with a simple matrix with the given spectrum, and changes the entries slightly to obtain the desired pattern of zero and nonzero entries while keeping the spectrum intact. In this application of the Jacobian method we start with a diagonal matrix with the given spectrum, and change the off-diagonal entries one row and one column at a time, and adjust the diagonal entries. The process of showing that changing off-diagonal entries can be compensated with a change in the diagonal entries so that the spectrum remains intact involves using the Implicit Function Theorem (IFT). Checking the necessary conditions of the IFT for this problem involves a notion of robustness and genericity for the solution of the inverse problem we are interested in. A stronger version of this robustness is introduced in~\cite{sap15} as the Strong Spectral Property (SSP). In this section we introduce a property for finite matrices, similar to the SSP, which we call the Weak Spectral Property (WSP), and study some of its properties that are critical to our approach in solving the inverse problem.

\begin{definition}\label{D:PropertyP} 
	A symmetric matrix $A$ is said to have the \emph{Weak Spectral Property} (or $A$ has WSP for short) if $X=O$ is the only  symmetric matrix satisfying
	\begin{enumerate}[(1)]
	\item $X\circ I=O$, and \label{I:zero diagonal}
	\item $[X, A]=O$. \label{I:zero commutator}
	\end{enumerate}
\end{definition}
Note that any $1\times 1$ matrix has the WSP. When $n\ge 2$, an $n\times n$ scalar matrix cannot have the WSP, since any such matrix belongs to the center of the algebra of matrices. More generally, any matrix $A$ with the WSP cannot have a constant diagonal, because if all the diagonal entries of $A$ are equal to a constant $c$, then 
	\[ 	(A-cI)\circ I = O \text{ and }  [A-cI, A] = O. 	\]
Thus $X=A-cI$ is a nonzero solution of \eqref{I:zero diagonal} and \eqref{I:zero commutator} if $A$ is not a scalar matrix.
For $2\times 2$ matrices this necessary condition for the WSP is sufficient as well, and using this it is easy to give explicit examples to show that the WSP is not necessarily invariant under a change of basis.

\begin{obs}
A $2\times 2$ symmetric matrix has the WSP  if and only if it has distinct diagonal entries.
\end{obs}
However for $n\times n$ matrices with $n \geq 2$ having distinct diagonal entries does not guarantee the WSP, as the following example illustrates.

\begin{example} Consider the matrices
	\[
	A=
	\begin{bmatrix}
	4&0&1\\
	0&3&0\\
	1&0&2
\end{bmatrix}
	\text{ and }
	B=
	\begin{bmatrix}
	4&1&0\\
	1&3&1\\
	0&1&2
\end{bmatrix}.
	\]
Then a short calculation shows that the matrix $A$ has the WSP but not the matrix $B$. 
\end{example}
The following lemma shows that the WSP is an \emph{open property}. In fact the proof shows that \emph{any} sufficiently small perturbation, symmetric or not, of a matrix with the WSP satisfies the conditions of Definition~\ref{D:PropertyP}.

\begin{lemma}[The WSP is an open property]\label{L:Property P is open} If a matrix $A$ has the WSP, then any sufficiently small symmetric perturbation of $A$ also has the WSP. 
\end{lemma}
\begin{proof}
We prove the contrapositive of the first statement. Assume that for each $n\in \mathbb{N}$ there exist nonzero symmetric matrices $E_n$ and $X_n$ such that 
	\begin{enumerate}[(1)]
	\item $X_n\circ I = O$,\label{I:diagonal relation}
	\item $[X_n, A+E_n] = O$, and\label{I:Commutation relation}
	\item $\|E_n\| = 1/n$.
	\end{enumerate}
With a scaling we can set $\|X_n\|=1$ for all values of $n$, and $X_n$ still would satisfy the properties \eqref{I:diagonal relation} and \eqref{I:Commutation relation}. Note that the set of matrices of norm one is compact in the set of all square matrices with the Euclidean topology. Hence we can substitute the sequence $\{X_n\}$, if necessary, with a convergent subsequence that we again denote by $\{X_n\}$. Hence, taking the limit of the commutation relation \eqref{I:Commutation relation} as $n\to \infty$ we obtain
	\[
	[\lim_{n\to \infty} X_n, A]
	=
	\lim_{n\to \infty}
		[X_n, A]
	=
	\lim_{n\to \infty}
		[E_n, X_n] = O.
	\]
This proves that $A$ does not have the WSP because the nonzero matrix $X=\displaystyle\lim_{n\to \infty} X_n$ satisfies the conditions of Definition~\ref{D:PropertyP}. 
\end{proof}

\begin{obs}
Note that a matrix $A$ has the WSP if and only if $A+cI$ has the WSP for any real number $c$. This implies that the set of symmetric matrices without the WSP has no isolated points.
\end{obs}

We conclude this section by proving that the direct sum of two matrices with the WSP which do not share an eigenvalue has the WSP.

\begin{lemma}
Let $A$ and $B$ be two matrices with the WSP which do not have a common eigenvalue. Then $A\oplus B$ also has the WSP.
\end{lemma}
\begin{proof}
Let \[ X = \left[ \begin{array}{c|c}
X_1 & X_2 \\\hline X_3 & X_4
\end{array} \right], \] where $X_1$ and $X_4$ have the same size as $A$ and $B$, respectively, and assume that $X\circ I = O$, that is, $X_1 \circ I = O$ and $X_4 \circ I = O$. Furthermore, assume that $[ X, A \oplus B] = O$. We want to show that $X = O$. Note that 
\[ [X, A \oplus B] = \left[ \begin{array}{c|c}
[X_1,A] & X_2 B - A X_2 \\\hline X_3 A - B X_3 & [X_4, B]
\end{array} \right] = O. \]

Hence $[X_1,A] = O$ and $[X_4, B] = O$. Since $A$ and $B$ have the WSP, we conclude that $X_1 = O$ and $X_4 = O$. Furthermore, $X_2 B - A X_2 = O$ and $X_3 A - B X_3 = O$. That is, $X_2$ and $X_3$ are the intertwining matrices of $A$ and $B$. By Lemma 1.1 of \cite{lambdamu13} $X_2 = O$ and $X_3 = O$, since $A$ and $B$ do not have a common eigenvalue.
\end{proof}

We will use the following corollary of the above lemma in the next section.

\begin{corollary}\label{L:Property P with direct sum}
If a matrix $A$ has the WSP, then so does $A\oplus [c]$ where $c$ is any real number that is not an eigenvalue of $A$.
\end{corollary}

\section{Finite $\lambda$-SIEP solved with induction}\label{SectionFiniteLambdaSISP}
In this section, to control the norm of matrices constructed, we provide an inductive proof for Theorem \ref{finitelambdasiep} below, which is also proved using the Jacobian method in~\cite{lambdamu13}.

Let $G$ be a given graph on $n$ vertices $1, 2, \dots, n$ and with $m$ edges $\{ i_{\ell}, j_{\ell} \}$ where $\ell = 1,2,\ldots,m$, and $i_\ell < j_\ell$. Also, let $\lambda_1 < \lambda_2 < \ldots < \lambda_n$ be $n$ distinct real numbers, and let $\bm x = (x_1, x_2, \ldots, x_n)$ and $\bm y = (y_1, y_2, \ldots, y_m)$, where $x_i$'s and $y_i$'s are independent real unknowns. Define $M = M(\bm x, \bm y)$ to be the $n\times n$ symmetric matrix whose $i$-th diagonal entry is $x_i$,  its $(i_{\ell}, j_{\ell})$-th and $(j_{\ell}, i_{\ell})$-th entries are $y_\ell$, and it is zero elsewhere. Define \begin{align*}
f\colon \mathbb{R}^{n} \times  \mathbb{R}^{m} &\to \mathbb{R}^{n}\\
(\bm x, \bm y) &\mapsto \big(\lambda_1(M),\lambda_2(M),\ldots,\lambda_n(M)\big),
\end{align*}
where $\lambda_i(M)$ is the $i$-th smallest eigenvalue of $M$. We want to show that there is a real symmetric matrix $A$ whose graph is $G$ and its eigenvalues are $\lambda_1,\lambda_2,\ldots,\lambda_n$. In other words, we want to find $(\bm a, \bm b) \in \mathbb{R}^{n} \times  \mathbb{R}^{m}$ such that no entry of $\bm b$ is zero and $f(\bm a, \bm b) = (\lambda_1,\lambda_2,\ldots,\lambda_n)$. In order to do so we introduce a new function 
\begin{align*}
g\colon\mathbb{R}^{n} \times  \mathbb{R}^{m} &\to \mathbb{R}^{n}\\
(\bm a, \bm b) &\mapsto \left(\tr M \; , \; \frac{1}{2} \tr M^2 \; , \; \ldots \; , \; \frac{1}{n}\tr M^n \right).
\end{align*}

Note that by Newton's identities \cite{macdonald95} there exists an invertible function $h:\mathbb{R}^n \to \mathbb{R}^n$ such that $h\circ g(\bm a, \bm b) = f(\bm a, \bm b)$ for all $(\bm a, \bm b) \in \mathbb{R}^n \times \mathbb{R}^m $. Thus it suffices to show that there is $(\bm a, \bm b) \in \mathbb{R}^{n} \times  \mathbb{R}^{m}$ such that none of the entries of $\bm b$ is zero, and \begin{equation}\label{E:Newton function}
g(\bm a, \bm b) = \Bigg( \sum_{i=1}^n \lambda_i \;, \; \frac{1}{2}\sum_{i=1}^n \lambda_i^2 \;, \;\ldots \;,\; \frac{1}{n}\sum_{i=1}^n \lambda_i^n\Bigg).\end{equation}

Let $\jac_x(g)$ denote the matrix obtained from the Jacobian matrix of $g$ restricted to the columns corresponding to the derivatives with respect to $x_i$'s. Let $A = M(\bm a, \bm b)$ for some $(\bm a, \bm b) \in \mathbb{R}^n \times \mathbb{R}^m$. The matrix $\jac_x(g)\at_A$ is the evaluation of $\jac_x(g)$ at $(\bm a, \bm b)$. Then by Lemma 3.1 of \cite{lambdamu13} it is easy to see that
	\[
	\jac_x(g)\at_{A}
	 = 
	\left[ 
		\begin{array}{cccc}
		I_{11} & \cdots & I_{n-1,n-1} & I_{nn}\\
		A_{11} & \cdots & A_{n-1,n-1} & A_{nn}\\
		\vdots & \ddots & \vdots & \vdots\\
		A^{n-1}_{11} & \cdots & A^{n-1}_{n-1,n-1} & A^{n-1}_{nn}\\
		A^{n}_{11} & \cdots & A^{n}_{n-1,n-1} & A^{n}_{nn}
		\end{array} 
	\right],
	\]
where $A^k_{ij}$ denotes the $(i,j)$-th entry of $A^k$.

\begin{theorem} \label{finitelambdasiep}
Given a graph $G$ on $n$ vertices and a set of $n$ distinct real numbers $\Lambda=\{ \lambda_1, \lambda_2, \dots, \lambda_n \}$, there is an $n\times n$ symmetric matrix whose graph is $G$ and its spectrum is $\Lambda$.
\end{theorem}

\begin{proof}
By induction on the number of vertices we prove the stronger result that there exist $\bm a$ and $\bm b$ satisfying~\eqref{E:Newton function} such that the graph of $A = M(\bm a, \bm b)$ is $G$ and $A$ has the WSP.

When $n=1$, the matrix $A=[\lambda_1]$ has the WSP and establishes the base of the induction for the graph $G$ with one vertex. Now assume that our claim holds for $n-1$, that is, given a graph $G_{n-1}$ on $n-1$ vertices and $n-1$ distinct real numbers $\lambda_1, \lambda_2, \dots, \lambda_{n-1}$, there exists a symmetric matrix $A_{n-1}$ whose graph is $G_{n-1}$, its spectrum is $\{\lambda_1, \lambda_2, \dots, \lambda_{n-1}\}$, and $A_{n-1}$ has the WSP. To prove the claim for $n$, assume that the vertices of $G$  are labeled by $1, 2, \dots, n$ and let $G_{n-1}$ be the graph obtained from $G$ by removing its $n$-th vertex. The induction hypothesis applied to $G_{n-1}$ and $\lambda_1, \lambda_2, \dots, \lambda_{n-1}$ yields a matrix $A_{n-1}$ with the desired properties. Now, let 
	\begin{equation}\label{E:Definition of A}
		A = A_{n-1} \oplus [\lambda_n]
		=
		\left[
	\begin{array}{ccc|c}
	&&&\\
	&A_{n-1}& &\mathbf{0}\\
	&&&\\
	\hline
	&\mathbf{0}&& \lambda_n
	\end{array}
	\right].
	\end{equation} 

We will use the Implicit Function Theorem to make the desired entries in the last row and last column of $A$ nonzero, without changing the eigenvalues of the matrix. In order to do so, we show that $J = \jac_x(g)\at_A$ is nonsingular by proving that the rows of $J$ are linearly independent. Let $\bm \alpha = (\alpha_0, \alpha_1, \ldots, \alpha_{n-1})$. For $p(x)=\alpha_0+\alpha_1x+\dots+\alpha_{n-1}x^{n-1}$ and $X := p(A)$, it is clear that $[X, A]= O$. Assume that $\bm \alpha J = \bm 0$ and observe that this holds if and only if $X \circ I = O$.
To prove that $\bm \alpha=\bm 0$ let
	\[
	X=\renewcommand{\arraystretch}{1.25}
	\left[
	\begin{array}{ccc|c}
	&&&\\
	&X_{n-1}& &\mathbf{y}\\
	&&&\\\hline
	&\mathbf{y}^T&& 0
	\end{array} 
	\right].
	\]
Then
	\[
	[X,A]
	=
	\renewcommand{\arraystretch}{1.25}
	\left[
	\begin{array}{c|c}
	  \\[-1ex] \relax
	  [X_{n-1}, A_{n-1}] & \rotatebox[origin=c]{90}{$\lambda_n\mathbf{y}- A_{n-1}\mathbf{y}$}\\
	  \\[-1ex]
	  \hline
	  \lambda_n\mathbf{y}^T-\mathbf{y}^TA_{n-1}& 0 \\
	\end{array}
	\right]
	= O
	\]
implies $\lambda_n\mathbf{y}- A_{n-1}\mathbf{y} = \bm 0$. Since $\lambda_n$ is not an eigenvalue of $A_{n-1}$, we have $\mathbf{y}=\mathbf{0}$. On the other hand, $A_{n-1}$ has the WSP, and hence $[X_{n-1}, A_{n-1}]= O$ and $X_{n-1}\circ I= O$ imply that $X_{n-1}= O$. Therefore $X=p(A)= O$.

Finally, using the fact that $A$ has $n$ distinct eigenvalue and that its minimal polynomial has degree $n$ we conclude that $p(x)\equiv 0$, and hence $\bm \alpha=\mathbf{0}$ as we wanted.

Next, assume that the labeling of the vertices of $G$ is done in such a way that the $n$-th vertex is adjacent to vertices $1, 2, \dots, d$. Then by the Implicit Function Theorem, for sufficiently small $\varepsilon_1, \varepsilon_2, \dots, \varepsilon_d$, there exist numbers $\widetilde{A}_{11}, \widetilde{A}_{22}, \dots, \widetilde{A}_{nn}$ close to $A_{11}, A_{22}, \dots, A_{nn}$ such that if
	\[ 
	\widetilde{A} = 
	\left[ \def\arraystretch{1.5}\begin{array}{c|c}
	\arraycolsep=6pt\def\arraystretch{2}
	\begin{array}{cccc}
		\widetilde{A}_{11}& A_{12} & \cdots & A_{1,n-1}\\
		A_{21} & \widetilde{A}_{22} & \cdots & A_{2,n-1}\\
		\vdots & \vdots & \ddots & \vdots\\
		A_{n-1,1}& A_{n-1,2} & \cdots & \widetilde{A}_{n-1,n-1}	
	\end{array} & \def\arraystretch{1.12}
	\begin{array}{c}
		\varepsilon_1 \\ \vdots \\ \varepsilon_d \\ 0 \\ \vdots \\ 0
	\end{array}\\\hline
	\arraycolsep=8pt
	\begin{array}{cccccc}
		\varepsilon_1 & \cdots & \varepsilon_d & 0 & \cdots & 0 \end{array} &  \widetilde{A}_{nn}
	\end{array} \right ],
	\] 
	then \[g(\bm x, \bm y)\at_{\widetilde{A}} = \Bigg( \sum_{i=1}^n \lambda_i \;, \; \frac{1}{2}\sum_{i=1}^n \lambda_i^2 \;, \;\ldots \;,\; \frac{1}{n}\sum_{i=1}^n \lambda_i^n\Bigg).\] Furthermore, if $\varepsilon_i$'s are chosen to be nonzero, then $G$ is the graph of $\widetilde{A}$. Thus $f(\bm x, \bm y)\at_{\widetilde{A}} = (\lambda_1,\lambda_2,\ldots,\lambda_n)$, that is, the spectrum of $\widetilde{A}$ equals $\Lambda$. It is evident that this solution is not unique. The operator norm of the matrix $A$, denoted $\|A\|_{\textrm{op}}$, is defined by 
	\[ 
	\| A \|_{\textrm{op}} 
	= 
	\sup_{\| \bm v \|_2 = 1} \| A \bm v \|_2. 
	\]
Choosing $\varepsilon_i$'s small enough, we can have 
	\begin{equation}\label{E:Closeness condition}
	\|\widetilde{A}-A\|_{\text{op}}<\varepsilon
	\end{equation}
for any given $\varepsilon$. It only remains to prove that $\widetilde{A}$ has the WSP.  This is an immediate application of Corollary~\ref{L:Property P with direct sum} and Lemma \ref{L:Property P is open} and follows from Equations~\eqref{E:Definition of A} and \eqref{E:Closeness condition}.
\end{proof}

\begin{remark}
Note that the proof of Theorem \ref{finitelambdasiep} enables us to control the norm of $\widetilde{A}$ in each step. This will be important in examination of the infinite case of this SISP in the next section.
\end{remark}

\section{Infinite $\lambda$-SISP}\label{SecctionInfiniteLambdaSISP}
In this section we will prove an analogue of Theorem \ref{finitelambdasiep} for countably infinite (but not necessarily locally finite) graphs. This is done by taking the limit, in a suitable sense, of the matrices $\widetilde{A}$ that were constructed in the proof of Theorem \ref{finitelambdasiep}. Then we will show that this limit has the desired properties.

We shall need a corollary of the following result about a ``continuity property'' of the spectrum in the proof of our main theorem (See Chapter V, Theorem 4.10 of \cite{kato95}).
\begin{theorem}\cite[p. 291]{kato95}
Let $S$ and $T$ be bounded self-adjoint operators on a Hilbert space with the spectra $\sigma(S)$ and $\sigma(T)$, respectively. Then the Hausdorff distance $d_H(\sigma(S), \sigma(T))$ satisfies 
	\[
	d_H(\sigma(S), \sigma(T))\le \|S-T\|_{\textrm{op}}.
	\]
\end{theorem}
This theorem immediately implies the following corollary.
\begin{corollary}\label{C:Closeness of spectra}
Let $\{T_n\}_{n=1}^\infty$ be a sequence of bounded self-adjoint operators on a Hilbert space $\mathcal{H}$, and assume that $T_n\to T$ in the operator norm. Then for any $\lambda\in \sigma(T)$, every neighborhood of $\lambda$ intersects $\sigma(T_n)$ nontrivially provided that $n$ is sufficiently large.
\end{corollary}
Now we are ready to state and prove our main theorem.

\begin{theorem}[SISP with data $(G, \Lambda)$]\label{T:infinite bounded lambdasisp}
	Given an infinite graph $G$ on countably many vertices and a compact, infinite set of real numbers $\Lambda$, there exists a self-adjoint operator $T$ on the Hilbert space $\ell^2(\mathbb{N})$ such that 
	\begin{enumerate}[(i)]
	\item the (approximate point) spectrum of $T$ equals $\Lambda$, and
	\item the (real symmetric) standard matrix of $T$ has graph $G$.
	\end{enumerate}
\end{theorem}
\begin{proof} Let $\{\lambda_1, \lambda_2, \dots\}$ denote a countable dense subset of $\Lambda$. Suppose the vertices of $G$ are labeled by $\mathbb{N}$ and let $\widetilde{A}_n$ be the matrix obtained from applying Theorem~\ref{finitelambdasiep} to the finite graph $G[\{1, 2, \dots, n\}]$, the induced subgraph of $G$ on the first $n$ vertices, and the finite set of distinct real numbers $\{\lambda_1, \lambda_2, \dots, \lambda_n\}$. For each $n$ define the bounded linear operator $T_n$ on the Hilbert space of square-summable sequences $\ell^2(\mathbb{N})$ (denoted $\ell^2$ hereafter for short) such that 
	\[
	M_n =
	\left[
	\begin{array}{c|c}
	\begin{array}{ccc}
	&&\\
	&\widetilde{A}_n&\\
	&&\
	\end{array} & O\\\hline
	O & \begin{array}{ccc}
	\lambda_{n+1} &&\\
	&\lambda_{n+2}&\\
	&&\ddots
	\end{array}
	\end{array}
	\right]
	\]
is the matrix representation of $T_n$ with respect to the standard Hilbert basis $\mathfrak{B}=\{\bm {e_1}, \bm {e_2}, \dots\}$ of $\ell^2$, where the entries of $M_n$ that are not shown are zero. By Equation~\eqref{E:Closeness condition}, we can find $\widetilde{A}_{n+1}$ such that  
	\begin{equation}\label{E:norm-difference}
	\|
		\widetilde{A}_{n}\oplus[\lambda_{n+1}]
		-\widetilde{A}_{n+1}
	\|_{\textrm{op}}
	<
	\frac{1}{2^n}
	\end{equation}
and consequently	 for any $i\in \mathbb{N}$
	\[
	\|M_n \bm{e_i}-M_{n+1}\bm {e_i}\|_2<\frac{1}{2^n}.
	\]
Thus, the sequence of partial sums $\{M_n\bm {e_i}\}_{n=1}^\infty$ given by
	\[
	M_{n} \bm {e_i} = M_1\bm {e_i} +\sum_{k=1}^{n-1} (M_{k+1}\bm {e_i}-M_k\bm {e_i})
	\]
is absolutely convergent in $\ell^2$. Let $M$ denote the matrix whose columns are obtained by this limiting process, that is, $M$ is the matrix that $M\bm{e_i}=\lim_{n\to \infty} M_n \bm{e_i}$ for each $i\in \mathbb{N}$. Note that for each $n = 1,2,\dots$ the graph of $\widetilde{A}_n$ is the induced subgraph of $G$ on the first $n$ vertices. Thus, by construction $G$ is the graph of $M$. Our next objective is showing that $M$ is indeed the standard matrix of a bounded linear operator $T\colon \ell^2\to \ell^2$. 
Observe that
\begin{align*}
\|T_n - T_{n+1}\|_{\textrm{op}} 
&=
\sup_{\|\bm v\|_2=1} 
\|T_n \bm v - T_{n+1} \bm v\|_2\\
&= \sup_{\|\bm v\|_2=1}
\left\| 
 \left[ \begin{array}{c}
\left[ 
	\begin{array}{c}
	\widetilde{A}_n
	\left[ 
	\begin{array}{c}
	v_1 \\ \vdots \\ v_n
	\end{array} 
\right] 
\\ 
\hline
\lambda_{n+1} v_{n+1}
\end{array} \right] - \widetilde{A}_{n+1} \left[ \begin{array}{c}
v_1 \\ \vdots \\ v_{n+1}
\end{array} \right]\\ \hline
\begin{array}{c}
0 \\ \vdots \\ 0
\end{array}
\end{array} 
\right] 
\right\|_2  
\\
&= 
\sup_{\|\bm v\|_2=1} 
	\left\|
	\left( \left[ \begin{array}{c|c}
\widetilde{A}_n & \\ \hline
& \lambda_{n+1}
\end{array} \right] - \widetilde{A}_{n+1}
	\right) 
\left[ 
\begin{array}{c}
v_1 \\ \vdots \\ v_{n+1}
\end{array} 
\right] 
\right\|_2
\\
&< \frac{1}{2^n},
\end{align*}
where in the last line we have used the submultiplicative property of the operator norm and \eqref{E:norm-difference}. This inequality immediately implies that the sequence of partial sums $\{T_n\}_{n=1}^\infty$ given by
	\[
	T_n=T_1 +\sum_{k=1}^{n-1} (T_{k+1}-T_k)
	\]
is absolutely convergent in the Banach space of bounded operators $\mathcal{B}(\ell^2)$ and hence convergent to a limit $T$ with respect to the operator norm. Since for each $i\in \mathbb{N}$ we have $T \bm e_i =  \lim_{n\to \infty} T_n \bm e_i$ and $T_n \bm e_i=M_n\bm e_i$, we conclude that $T\bm e_i=M\bm e_i$ and thus $M$ is the standard matrix of $T$. 

It remains to prove that $\sigma(T)=\Lambda$.  First, we claim that each $\lambda_i \in \{\lambda_1, \lambda_2, \dots, \}\subset\Lambda$ is in the spectrum of $T$, that is, $T-\lambda_i I$ is not invertible. To see this note that as $n\to \infty$
\[
\|(T-\lambda_i I)-(T_n-\lambda_i I)\|_{\textrm{op}}=\|T-T_n\|_{\textrm{op}}\to 0
\]
which shows the existence of noninvertible operators in any neighborhood of $T-\lambda_i I$. Since invertibility is an open property in any unital Banach algebra and in particular in $\mathcal{B}(\ell^2)$, we have $\{\lambda_1, \lambda_2, \dots\}\subset \sigma(T)$, and the claim is proved. This inclusion implies $\Lambda \subset \sigma(T)$, because $\{\lambda_1, \lambda_2, \dots\}$ is dense in $\Lambda$ and $\sigma(T)$ is closed in $\mathbb{R}$.

Next, since $T_n\to T$ in the operator norm and $\sigma(T_n)=\Lambda$ for all $n$, by Corollary~\ref{C:Closeness of spectra} we conclude that for any $\lambda\in \sigma(T)$, every neighborhood of $\lambda$ intersects $\Lambda$. Hence the reverse inclusion $\sigma(T)\subset \Lambda$ is also established. 

Finally, to complete the proof of the theorem note that the spectrum of any self-adjoint operator equals its approximate point spectrum. 
\end{proof}

The isolated points of the spectrum of any self-adjoint operator are its eigenvalues, for instance, by an application of Gelfand's continuous functional calculus~\cite{douglas98}.  Thus the isolated points of $\Lambda$ are contained in the point spectrum of any solution $T$--obtained as in the proof of Theorem~\ref{T:infinite bounded lambdasisp} or otherwise--for the SISP with data $(G,\Lambda)$. On the other hand, by focusing only on the solutions $T$ obtained as in the proof of Theorem~\ref{T:infinite bounded lambdasisp} we will prove that the set of limit points of the spectrum of $T$ equals the \emph{essential spectrum} of $T$. This will allow us to show that our solutions of every SISP with the same $\Lambda$ are \emph{approximately unitarily equivalent}. We will then use this equivalence relation to show that the multiplicity of isolated eigenvalues of the constructed solutions is exactly one. 

\begin{definition}
Let $T$ be an operator on a Hilbert space. Then the \emph{essential spectrum} of $T$,  $\sigma_{\textrm{ess}}(T)$, is defined by
	\begin{equation}\label{E:essential via Fredholm}
	\sigma_{\textrm{ess}}(T)=\{\lambda\in \mathbb{C}\mid T-\lambda I \text{ is not a Fredholm operator}\}.
	\end{equation}
If $T$ is self-adjoint, the essential spectrum of $T$ can be equivalently defined as the complement of its discrete spectrum. That is,
	\begin{equation}\label{E:essential via discrete}
	\sigma_{\textrm{ess}}(T)=\sigma(T)\setminus \sigma_{\textrm{discr}}(T),
	\end{equation}
where the discrete spectrum, denoted $\sigma_{\textrm{discr}}(T)$, is the set of isolated eigenvalues of $T$ with finite multiplicity.
\end{definition}

There is a famous classification result due to Weyl, von Neumann, and Berg (Theorem \ref{T:WvB theorem} below) involving the essential spectrum of self-adjoint operators and the multiplicity of their isolated eigenvalues. To state it we need the following definition.
\begin{definition} Two bounded operators $T_1$ and $T_2$ on Hilbert spaces $\mathcal{H}_1$ and $\mathcal{H}_2$ are said to be \emph{approximately unitarily equivalent} (written $T_1\sim_a T_2$) if there is a sequence of unitary isomorphisms of Hilbert spaces $U_n\colon \mathcal{H}_1\to \mathcal{H}_2$ such that $T_1=\lim_{n\to \infty} U_n^*T_2U_n$.
\end{definition}
It is well-known (see for instance~\cite{davidson96}) that if $T_1\sim_a T_2$, then a sequence of unitary operators $\{U_n\}_{n=1}^\infty$ can be chosen such that in addition to $T_1=\lim_{n\to \infty} U_n^*T_2U_n$, one also has that $T_1-U_n^*T_2U_n$ belongs to the ideal $\mathcal{K}(\mathcal{H}_1)$ of compact operators.
\begin{theorem}[Weyl--von Neumann--Berg, {see \cite[Theorem II.4.4]{davidson96}}]\label{T:WvB theorem}
Two self-adjoint operators $M$ and $N$ on separable Hilbert spaces are approximately unitarily equivalent if and only if 
\begin{enumerate}[(i)]
\item $\sigma_{\textup{ess}}(M)=\sigma_{\textup{ess}}(N)$, and
\item $\dim\ker (M-\lambda I)=\dim\ker (N-\lambda I)$ for all $\lambda\in \mathbb{C}\setminus \sigma_{\textup{ess}}(M)$.
\end{enumerate}
\end{theorem}

Now we are ready to prove the following corollaries of Theorem \ref{T:infinite bounded lambdasisp} concerning the limit points of the spectrum and the approximate unitary equivalence of any two constructed solutions of a given SISP.

\begin{corollary}\label{C:limit points of spectrum}
Let $T$ be an operator obtained according to the proof of Theorem~$\ref{T:infinite bounded lambdasisp}$ as a solution to the SISP with data $(G, \Lambda)$. Then $\sigma_{\textup{ess}}(T)=\Lambda'$, where $\Lambda'$ is the set of limit points of $\Lambda$.
\end{corollary}

\begin{proof}
Let $T=\lim_{n\to \infty} T_n$ where $\{T_n\}_{n=1}^\infty$ is the sequence of operators defined as in the proof of Theorem~\ref{T:infinite bounded lambdasisp} relative to the countable dense subset $\{\lambda_1, \lambda_2, \dots\}$ of $\Lambda$. The containment $\Lambda'\subset \sigma_{\textrm{ess}}(T)$ is clear from~\eqref{E:essential via discrete} and the fact that $\Lambda=\sigma(T)$, as shown in Theorem~\ref{T:infinite bounded lambdasisp}.  
To prove the reverse inclusion, let $\lambda\in \sigma_{\textrm{ess}}(T)$ be arbitrary. By definition~\eqref{E:essential via Fredholm} and Atkinson's Theorem~\cite[Theorem~5.17]{douglas98} $T-\lambda I$ is not invertible ``modulo compact operators.'' More precisely, if we denote the ideal of compact operators in $\mathcal{B}(\ell^2)$ by $\mathcal{K}(\ell^2)$, then $(T-\lambda I)+\mathcal{K}(\ell^2)$ is not an invertible element of the \emph{Calkin algebra} $\mathcal{B}(\ell^2)/\mathcal{K}(\ell^2)$. Moreover, the equality
	\[
	T-\lambda I=
	[T-\operatorname{diag}(\lambda_1, \lambda_2, \dots)]
	+
	[\operatorname{diag}(\lambda_1, \lambda_2, \dots) -\lambda I]
	\]
implies that 
	\[
	(T-\lambda I)+\mathcal{K}(\ell^2)=\operatorname{diag}(\lambda_1-\lambda, \lambda_2-\lambda, \dots) +\mathcal{K}(\ell^2),
	\]	
because $T-\operatorname{diag}(\lambda_1, \lambda_2, \dots)$ is the limit of the finite-rank operators $T_n-\operatorname{diag}(\lambda_1, \lambda_2, \dots)$, and hence it is compact. To finish the proof of this part, we only need to observe that $\operatorname{diag}(\lambda_1-\lambda, \lambda_2-\lambda, \dots) +\mathcal{K}(\ell^2)$ is noninvertible in the Calkin algebra (if and) only if $0$ is a limit point of $\{\lambda_1-\lambda, \lambda_2-\lambda, \dots\}$. This is equivalent to saying that $\lambda$ is a limit point of $\{\lambda_1, \lambda_2, \dots\}$ in which case $\lambda \in \Lambda'$.
\end{proof}

We mention in passing that if $\Lambda'$  is a singleton, $\Lambda'=\{\lambda\}$, then $T-\lambda I$ is compact. To see this, let $T_n$ and $T$ be as in the proof of Theorem~\ref{T:infinite bounded lambdasisp}. If $\lambda$ is the only limit point of $\Lambda$, then all $T_n-\lambda I$ and consequently their limit in operator norm $T-\lambda I$ are compact operators.

\begin{corollary}\label{C:approximate unitary equivalence}
Let $S$ and $T$ be any two operators obtained according to the proof of Theorem~$\ref{T:infinite bounded lambdasisp}$ as solutions to the SISP with data $(G_1, \Lambda)$ and $(G_2, \Lambda)$, respectively. Then $S$ and $T$ are approximately unitarily equivalent. In particular, $T\sim_a \operatorname{diag}(\lambda_1, \lambda_2, \dots)$ for any countable dense subset $\{\lambda_1, \lambda_2, \dots\}$ of $\Lambda$.
\end{corollary}

\begin{proof}
Let $S_n$ and $T_n$ be defined as in the proof of Theorem~\ref{T:infinite bounded lambdasisp} such that $S=\lim_{n\to \infty} S_n$ and $T=\lim_{n\to \infty} T_n$. Then for all $n$, 
\begin{itemize}
\item $\sigma_{\textup{ess}}(S_n)=\sigma_{\textup{ess}}(T_n)=\Lambda'$, and 
\item $\dim\ker (S_n-\lambda I)=\dim\ker (T_n-\lambda I)=1$ for all $\lambda\in \Lambda\setminus \Lambda'$,
\end{itemize}
because, by construction, each of $S_n$ and $T_n$ can be realized as the direct sum of an $n\times n$ matrix and an infinite diagonal matrix for which verifying these two conditions is straightforward. Thus Theorem~\ref{T:WvB theorem} implies $S_n\sim_a T_n$. For each $n$, let $\{U_{n, k}\}_{k=1}^\infty$ denote the sequence of unitary operators that satisfy $S_n=\lim_{k\to \infty} U_{n, k}^*T_nU_{n, k}$ and in addition, by passing to a subsequence of $\{U_{n, k}\}_{k=1}^\infty$ if necessary, we can arrange that when $k=1$ we have
	\[
	\|S_n-U_{n, 1}^*T_nU_{n, 1}\|_{\text{op}}
	<\max
	\left\{
		\|T_n-T\|_{\text{op}}, \|S_n-S\|_{\text{op}}, \frac{1}{n}
	\right\}.
	\]
Define the sequence of unitary operators $\{V_n\}_{n=1}^\infty$ by setting $V_n=U_{n, 1}$. We wish to prove that $S=\lim_{n\to \infty} V_n TV_n^*$. This is accomplished by a $3\varepsilon$--argument as follows. Given $\varepsilon>0$, choose $N>0$ large enough such that
\[
\max
\left\{
	\|T_n-T\|_{\text{op}}, \|S_n-S\|_{\text{op}}, \frac{1}{n}
\right\}
<\varepsilon
\]
whenever $n>N$. Then for any such $n$,
	\begin{small}
	\begin{align*}
	\|S-V_n^*TV_n\|_{\text{op}}
	&\le
	\|S-S_n\|_{\text{op}}+\|S_n-V_n^*T_nV_n\|_{\text{op}}+\|V_n^* T_nV_n-V_n^* T V_n\|_{\text{op}}
	\\
	&=
	\|S-S_n\|_{\text{op}}+\|S_n-U_{n, 1}^*T_nU_{n, 1}\|_{\text{op}}+\|T_n-T\|_{\text{op}}
	\\
	&< 3\max
		\left\{
			\|T_n-T\|_{\text{op}}, \|S_n-S\|_{\text{op}}, \frac{1}{n}
		\right\}
	\\
	&<	3\varepsilon.
	\end{align*}
	\end{small}
Since $\varepsilon>0$ is arbitrary, our claim is proved.
\end{proof}
``Approximate unitary equivalence'' cannot be replaced by the stronger notion of ``unitary equivalence'' in Corollary~\ref{C:approximate unitary equivalence} as the following simple example shows. Let $G$ be the empty graph on countably many vertices and let $\Lambda$ be the compact set $\{\frac{1}{n}\mid n\in \mathbb{N}\}\cup\{0\}$. Then our construction in the proof of Theorem~\ref{T:infinite bounded lambdasisp} produces the following solutions for the SISP with data $(G, \Lambda)$:
	\begin{equation}\label{E:non-equivalent diagonal solutions}
	\operatorname{diag}\left(1, \frac{1}{2}, \frac{1}{3} \dots\right)
	\text{ and }
	\operatorname{diag}\left(0, 1, \frac{1}{2}, \dots\right)
	\end{equation}
corresponding to two different choices of countable dense subsets of $\Lambda$, namely $\Lambda \setminus \{0\}$ and $\Lambda$. The two operators in \eqref{E:non-equivalent diagonal solutions} are approximately unitarily equivalent but clearly not unitarily equivalent. 

\begin{remark} Corollary~\ref{C:approximate unitary equivalence} immediately implies that if $T$ is a solution of the SISP with data $(G_1, \Lambda_1)$ constructed as above, then any isolated eigenvalue of $T$ has multiplicity one. To see this it suffices to let $G_2$ be the empty graph on countably many vertices and let $\Lambda_2=\Lambda_1$ so that a solution to the SISP with data $(G_2, \Lambda_2)$ is the diagonal operator $S=\operatorname{diag}(\lambda_1, \lambda_2, \dots)$ for a countable dense subset $\{\lambda_1, \lambda_2, \dots\}$ of $\Lambda$. Since $S\sim_a T$ by Corollary~\ref{C:approximate unitary equivalence}, an application of Theorem~\ref{T:WvB theorem} to the pair of self-adjoint operators $S$ and $T$ gives 
	\[
	\dim\ker (T-\lambda I)=1 \text{ for all } \lambda\in \sigma(T)\setminus \sigma_{\textup{ess}}(T)=\Lambda\setminus \Lambda'.
	\]
\end{remark}
\section{Conclusion}
Motivated by some problems in science and engineering such as the inverse Sturm-Liouville problem, in this manuscript we have introduced a structured inverse spectrum problem (SISP) for infinite graphs and have shown the existence of a (non-unique) solution for this problem. We have used the Jacobian method to establish a procedure for controlling the norm of the solutions of the finite problem in a way that the constructed sequence of solutions for finite cases converges to a solution of the infinite problem. We have also shown that any two solutions of the SISP constructed by our method are approximately unitarily equivalent.

\section{Acknowledgement}
The research of the first author is supported by the Natural Sciences and Engineering Research Council of Canada.

\vfill
\bibliographystyle{plain}
\bibliography{ref151217}

\end{document}